\documentclass[11pt]{amsart}
\usepackage{amssymb,amsthm,amsmath}
\usepackage[numbers,sort&compress]{natbib}
\usepackage{color}
\usepackage{graphicx}
\usepackage{tikz}
\usepackage[colorlinks,linkcolor=blue,anchorcolor=blue,
citecolor=blue]{hyperref}
\usepackage{comment}

\date{}

\title[Dispersive estimate for Klein-Gordon equation]{
Dispersive estimate for quasi-periodic Klein-Gordon equation on 1-d lattices}

\author{Hongyu Cheng}\address{
School of Mathematical Sciences, Tiangong University\\
Tianjin, 300378, P.R. China} \email{hychengmath@tiangong.edu.cn}

\newtheorem{theorem}{Theorem}[section]
\newtheorem{proposition}{Proposition}[section]
\newtheorem{lemma}{Lemma}[section]
\theoremstyle{definition}
\newtheorem{remark}{Remark}[section]

\newcommand{\me}{\mathrm{e}}   
\newcommand{\mi}{\mathrm{i}}     
\newcommand{\ZZ}{\mathbb{Z}}
\newcommand{\NN}{\mathbb{N}}
\newcommand{\RR}{\mathbb{R}}

\newcommand{\CC}{\mathbb{C}}
\newcommand{\TT}{\mathbb{T}}

\numberwithin{equation}{section}

\newcommand{\R}{{\mathbb R}}
\newcommand{\T}{{\mathbb T}}

\newcommand{\Z}{{\mathbb Z}}

\begin{document}

\maketitle
\begin{abstract}
The dispersive estimate plays a pivotal role in establishing the long-term behavior of solutions to the nonlinear equation, thereby being crucial for investigating the well-posedness of the equation.
In this work we prove that the solutions to Klein-Gordon equation on 1-d lattices follow the dispersive estimate provided that potential is quasi-periodic with Diophantine frequencies and closed to positive constants.
\end{abstract}
\section{Introduction}\label{Introduction}
The nonlinear Klein-Gordon equations with quasi-periodic potential is given by
\begin{equation}\label{mainequation}
\partial_{t}^{2}u_{n}=\Delta u_{n}-V(\theta+n\omega)u_{n}
+\lambda u_{n}^{2\kappa+1},\kappa\geq1,
\end{equation}
where
\begin{equation*}
\Delta u_{n}:=u_{n+1}+u_{n-1}-2u_{n},
\end{equation*}
and will be called focusing and defocusing for $\lambda=1$ and $\lambda=-1,$ respectively. The linear equation is
\begin{equation}\label{mainequation+}
\partial_{t}^{2}u_{n}=\Delta u_{n}-V(\theta+n\omega)u_{n}.
\end{equation}
This model arises in many physical contexts, such as lattice that arises for an array of coupled torsion pendulums under the effect of gravity \cite{Pelinovskys20}, the local denaturation of the DNA double strand \cite{Peyrardb89}, the fluxon
dynamics in arrays of superconducting Josephson-junctions \cite{UCM93}.

The Klein-Gordon equation relates to Schr\"odinger equation in the sense
that the latter are the natural envelope wave reduction of the former \cite{Kivsharp92}.
In \cite{StefanovK05} and \cite{MielkeC10}, the authors Pelinovsky-Stefanov and Mielke-Patz show that the decay estimate of the discrete linear Schr\"odinger equation without potential is $(|t|+1)^{-1/3},\forall t\in\RR.$ We refer \cite{KomechKK06,PelinovskyS08,Boussaid08,KevrekidisPS09,Bambusi13,EgorovaKT15} and the references therein to the readers for other related works.
Later, in \cite{zhao20}, Bambusi-Zhao show the dispersive estimate of the discrete linear and nonlinear Schr\"odinger equation by assuming the potential $V$ is analytic small enough and quasi-periodic with Diophantine frequency $\omega \in \mathbb{R}^d,$ which is defined by
\begin{equation*}
\|\langle n, \omega\rangle\|_{\RR/\ZZ} \geq \frac{\gamma}{|n|^\tau}, \quad 0 \neq n \in \mathbb{Z}^d.
\end{equation*}
Subsequently, Mi-Zhao prove that the discrete periodic one-dimensional
Schr\"odinger operators also possess a similar dispersive estimate \cite{Mizh20,Mizh22}.

Recently, in \cite{wuyz24}, Wu-Yang-Zhou prove that the Schr\"odinger equation is global well-posedness in $\ell^{p}(1\leq p\leq\infty).$
While the Klein-Gordon equation represents different behavior, that only the
defocusing one is global well-posedness in $\ell^{p}(2\leq p\leq2\kappa)$ and, in contrast, the solutions of the focusing one with negative energy blow up within a finite time. As a result, whether the Klein-Gordon equations inherit the subtle dispersive estimate is an interesting question.

Set
\begin{equation*}
C_{r}^{\omega}(\TT^{d},*)=\big\{ P: \TT^{d}\rightarrow *:\
|P|_{r}:=\sup_{|{Im}\theta|\leq r} |P(\theta)|
\big\},
\end{equation*}
where $*$ will usually denote $\RR, sl(2,\RR), SL(2,\RR)$. We denote $C^{\omega}(\TT^{d},*)=\cup_{h>0}C^{\omega}_{h}(\TT^{d},*).$
Without lose of generality, we set the potential as
\begin{equation}\label{202310230}
V(\theta)=1+P(\theta)
\end{equation}
and assume that $P(\cdot)\in C_{r}^{\omega}(\TT^{d},\RR)$. Moreover,
we denote $D C_{d}(\gamma, \tau)(\gamma>0, \tau>d-1)$ by the set of Diophantine vectors.
Our first main result is the following.

\begin{theorem}\label{theorem1}
Consider the linear Klein-Gordon equation \eqref{mainequation+} with the potential $V$ defined by \eqref{202310230} and assume $\omega\in D C_{d}(\gamma, \tau), P\in C_{r}^{\omega}(\TT^{d},\RR),\ \theta\in\TT^{d}.$ Then there exist $\varepsilon_{*}=\varepsilon_{*}(r,\gamma,\tau,d)>0,$ and an absolute constant $a,$ such that if $|P|_{r}=\varepsilon_{0}\leq\varepsilon_{*},$ any solutions $u(t)$ to the linear Klein-Gordon equation \eqref{mainequation+}
with $u(0),\partial_{t}u(0)\in\ell^{1}$
enjoy the following energy and dispersive estimates, $\forall t\in\R,$
\begin{equation}\label{mainestimate}
\begin{split}
|u(t)|_{\ell^{2}}^{2}&\leq (1+2\varepsilon_{0})\sum_{n\in\ZZ}
\big((5+2\varepsilon_{0})u_{n}(0)^{2}+(\partial_{t}u_{n}(0))^{2}\big),\\
|u(t)|_{\ell^{\infty}}&\leq 22872\langle t\rangle^{-1/3}
|\ln\varepsilon_{0}|^{a(\ln\ln(2+\langle t\rangle))^{2}d}
(|u(0)|_{\ell^{1}}+2|\partial_{t}u(0)|_{\ell^{1}}).
\end{split}
\end{equation}
\end{theorem}

As for the non-linear system \eqref{mainequation}, we also have the following estimates.

\begin{theorem}\label{theorem2}
Consider \eqref{mainequation} with $\kappa>5,$ assume that the hypotheses in Theorem~\ref{theorem1} hold, and fix $(\kappa-2)^{-1}<\zeta<1/3,$ then
there exist $K=K(\varepsilon_{0},\zeta)$ and
$\delta_{*}>0,$ with $\delta_{*}=\delta(r,\gamma,\tau,d,\varepsilon_{0},\zeta)$ such that, for any $\theta\in\TT^{d}$ and any $t\in\RR,$ if $|u(0)|_{\ell^{1}},|\partial_{t}u(0)|_{\ell^{1}}\leq \delta_{*},$
the solutions to \eqref{mainequation} enjoy the following estimates:
\begin{equation*}
\begin{split}
|u(t)|_{\ell^{2}}^{2}&\leq (1+2\varepsilon_{0})\sum_{n\in\ZZ}
\big(2(3+\varepsilon_{0})u_{n}(0)^{2}+(\partial_{t}u_{n}(0))^{2}\big),\\
|u(t)|_{\ell^{\infty}}&\leq K\langle t\rangle^{-\zeta}
(|u(0)|_{\ell^{1}}+|\partial_{t}u(0)|_{\ell^{1}}).
\end{split}
\end{equation*}
\end{theorem}

Before concluding this section, let us analyze the philosophy of this work.
Consider the Klein-Gordon equation \eqref{mainequation+}, which can be rewritten as the first order system %
\begin{equation}\label{mainequations}
\left\{\begin{aligned}
\partial_{t}u_{n}&=(H_{\theta}+3)^{1/2}w_{n},\\
\partial_{t}w_{n}&=-(H_{\theta}+3)^{1/2}u_{n},
\end{aligned}\right.
\end{equation}
where
\begin{equation}\label{1.1}
\left(H_{\theta} x\right)_n=-(x_{n+1}+x_{n-1})+P(\theta+n\omega) x_n, n \in \mathbb{Z}.
\end{equation}
Introduce the complex variables $q=(q_{n})_{n\in\ZZ}$ and $\overline{q}=(\overline{q}_{n})_{n\in\ZZ},$ by
\begin{equation}\label{complexvariables}
\begin{split}
\left(
\begin{matrix}
q_{n}
\\
\overline{q}_{n}
\end{matrix}
\right)=M\left(
\begin{matrix}
u_{n}
\\
w_{n}
\end{matrix}
\right),
\end{split}
\end{equation}
where
\begin{equation}\label{definem}
\begin{split}
M=\frac{1}{\sqrt{2}}\left(
\begin{matrix}
1,\ -\mi
\\
1,\ \ \mi
\end{matrix}
\right),\quad
M^{-1}=\frac{-\mi}{\sqrt{2}}\left(
\begin{matrix}
\mi,\ \ \mi
\\
-1,\ 1
\end{matrix}
\right).
\end{split}
\end{equation}
Then, \eqref{mainequations} is changed into %
\begin{equation}\label{mmainequationss}
\begin{aligned}
\partial_{t}q_{n}=&\mi(H_{\theta}+3)^{1/2}q_{n},\\
\partial_{t}\overline{q}_{n}=&-\mi(H_{\theta}+3)^{1/2}\overline{q}_{n}.
\end{aligned}
\end{equation}
The functional calculus shows that $x=(x_{n})_{n\in\ZZ}$ is the eigenvector of $H_{\theta}$ defined by \eqref{1.1} with eigenvalue $E\in\Sigma_{\theta}\subset[-2-\max_{\theta\in\TT} |P(\theta)|,2+\max_{\theta\in\TT} |P(\theta)|]$ if and only if $x=(x_{n})_{n\in\ZZ}$ is the eigenvector of $(H_{\theta}+3)^{1/2}$ with eigenvalue $\sqrt{E+3}.$
This relation inspires us to investigate the spectrum of the quasi-periodic linear operator $H_{\theta}.$

To obtain the delicate dispersive estimate, the researchers Bambusi-Zhao \cite{zhao20} provide a lower bound for $|\partial_{\rho_{J}}^{j}E|(j=2,3)$ by constructing a quantitative KAM reducibility involving intricate estimations.
Thus, we have to bound $|\partial_{\rho_{J}}^{j}\sqrt{E+3}|(j=2,3)$ below. To this end, new technique is needed to modify the estimates of quantitative KAM reducibility theorem and the related subsequent calculations, which is the main work of this study, we refer the readers to the proof Theorem~\ref{mainresult} and lemmata~\ref{resonantcase}, \ref{gamma0estimate} for details.

\section{The result of a KAM procedure}
\subsection{Schr\"odinger cocycle}
For the Schr\"odinger operator $H_{\theta}$ defined by \eqref{1.1}, we consider its
eigenvalue problem $H_{\theta}x=Ex$ and get the Schr\"odinger cocycle $(\omega, A_{0}+F_{0}):$
\begin{equation}\label{202402070}
\begin{split}
\left(
\begin{matrix}
x_{n+1}
\\
x_{n}
\end{matrix}
\right)=(A_{0}+F_{0}(\theta+n\omega))
\left(
\begin{matrix}
x_{n}
\\
x_{n-1}
\end{matrix}
\right)
\end{split}
\end{equation}
with
\begin{equation*}
\begin{split}
A_{0}=\left(
\begin{matrix}
-E,\ -1
\\
1,\quad 0
\end{matrix}
\right),\
F_{0}(\cdot)=
\left(
\begin{matrix}
P(\cdot), 0
\\
0, \ \  0
\end{matrix}
\right).
\end{split}
\end{equation*}
The evolutions of \eqref{mainequation} and \eqref{mainequation+} depend heavily on the dynamics of the cocycle $(\omega, A_{0}+F_{0}).$ We will get the dispersive estimate by applying a quantitative KAM reducibility of this cocycle.

\subsection{A quantitative KAM reducibility}
In this subsection, we will construct a quantitative KAM reducibility to deal with the Schr\"odinger cocycle $(\omega, A_{0}+F_{0})$ defined by \eqref{202402070}. Set
\begin{equation*}
\begin{split}
\sigma=1/200,\ |P|_{r}:=\varepsilon_{0}\leq\varepsilon_{*},\ \varepsilon_{j+1}=\varepsilon_{j}^{1+\sigma},\
N_{j}=4^{j+1}\sigma|\ln\varepsilon_{j}|, j\geq0.
\end{split}
\end{equation*}

\begin{theorem}\label{mainresult}
Consider the Schr\"odinger cocycle $(\omega, A_{0}+F_{0})$ defined by \eqref{202402070}.
There exists $\varepsilon_{*}=\varepsilon_{*}(\gamma,\tau,r,d)>0$ such that if
$\varepsilon_{0}:=|P|_{r}\leq\varepsilon_{*},$ the following conclusions hold:

$(1):$ For
any $j\in\NN,$ we have the following:

$(1.1):$ There exists a Borel set $\Sigma_{j}\subset\Sigma,$ with
$\{\Sigma_{j}\}_{j\in\NN}$ mutually disjoint, satisfying
\begin{equation}\label{202308130}
\begin{split}
\big|&\rho(\Sigma_{j+1})\big|\leq3|\ln\varepsilon_{j}|^{2d} \varepsilon_{j}^{\sigma},\ j\geq0,\\
\big|&\Sigma\setminus\widetilde{\Sigma}\big|=0,\ \
\widetilde{\Sigma}=\cup_{j\geq0}\Sigma_{j}.
\end{split}
\end{equation}
Moreover, the Schr\"odinger cocycle $(\omega,A_{0}+F_{0})$ is reducible on $\widetilde{\Sigma}.$ More precisely, there exists
$Z\in C^{\omega}(\widetilde{\Sigma}\times2\TT^{d},SL(2,\RR))$ with
$B\in C^{1}(\widetilde{\Sigma},SL(2,\RR))$ such that
\begin{equation}\label{202308131}
\begin{split}
Z(\cdot+\omega)^{-1}(A_{0}+F_{0}(\cdot))Z(\cdot)=B
\end{split}
\end{equation}
with the estimates
\begin{equation}\label{202308132}
\begin{split}
|B-A_{0}|_{C_{W}^{1}(\Sigma_{0})}\leq\varepsilon_{0}^{1/3},\quad
|B|_{C_{W}^{1}(\Sigma_{j+1})}\leq N_{j}^{10\tau}, j\geq0.
\end{split}
\end{equation}

$(1.2):$ The eigenvalues of $B|_{\Sigma_{j}}$ are of the form $\me^{\pm \mi\xi}$ with $\xi\in\RR,$ and, for every $j\geq0,$ there is $k_{j}:\widetilde{\Sigma}\rightarrow\ZZ^{d},$ such that
\begin{equation}\label{202308133}
\begin{split}
0<|k_{j}|\leq N_{j} \ \text{on} \ \Sigma_{j+1}, \text{and}
\ k_{l}=0 \ \text{on}\  \Sigma_{j} \ \text{for}\ l\geq j,
\end{split}
\end{equation}
and, set $\langle k\rangle_{\omega}:=\frac{1}{2}\langle k,\omega\rangle,$
\begin{equation}\label{202308134}
\begin{split}
\xi=\rho-\sum_{l\geq0}\langle k_{l}\rangle_{\omega} \ \text{and}\
0<|\xi|_{\Sigma_{j+1}}<2\varepsilon_{j}^{\sigma}.
\end{split}
\end{equation}

$(2):$ Given any $J\in\NN,$ for
$0\leq j\leq J,$ there exists $\Gamma_{j}^{(J)}\subset[\inf\Sigma,\sup\Sigma],$ satisfying:

$(2.1):$
$\{\Gamma_{j}^{(J)}\}_{j=0}^{J}$ are mutually disjoint and
$\overline{\cup_{0\leq j\leq J}\Gamma_{j}^{(J)}}=[\inf\Sigma,\sup\Sigma],$
and $\cup_{0\leq j\leq J}\Gamma_{j}^{(J)}$ consists of at most $|\ln\varepsilon_{0}|^{2J^{2}d}$ open intervals. Moreover,
$\Sigma_{j}\subset \Gamma_{j}^{(J)},\ \text{for}\ 0\leq j\leq J (J\geq1),$
and $|\rho(\Sigma_{j+1}^{(J)})|\leq 3|\ln\varepsilon_{j}|^{2d}\varepsilon_{j}^{\sigma}, 0\leq j\leq J-1.$

$(2.2):$ For $0\leq j\leq J,$ there exist
\begin{equation}\label{202308135}
\begin{split}
A_{J}&\in C^{3}(\Gamma_{j}^{(J)},SL(2,\RR)),\ \
F_{J}\in C^{3}(\Gamma_{j}^{(J)}\times \TT^{d},gl(2,\RR)),\\
Z_{J}&\in C^{3}(\Gamma_{j}^{(J)}\times 2\TT^{d},SL(2,\RR))
\end{split}
\end{equation}
such that
\begin{equation}\label{202308136}
\begin{split}
Z_{J}(\cdot+\omega)^{-1}(A_{0}+F_{0}(\cdot))Z_{J}(\cdot)=A_{J}+F_{J}(\cdot)
\end{split}
\end{equation}
with the estimates
\begin{equation}\label{202308137}
\begin{split}
&|F_{J}|_{C^{3}(\Gamma_{j}^{(J)}),\TT^{d}}\leq\varepsilon_{J},\ 0\leq j\leq J,\\
&|A_{J}-A_{0}|_{C^{3}(\Gamma_{0}^{(J)})}\leq\varepsilon_{0}^{1/2},\ \
|Z_{J}-Id|_{C^{3}(\Gamma_{0}^{(J)}),2\TT^{d}}\leq\varepsilon_{0}^{1/3}.
\end{split}
\end{equation}
If $J\geq1,$ then for $0\leq j\leq J-1,$
\begin{equation}\label{202308138}
\begin{split}
|A_{J}|_{C^{3}(\Gamma_{j+1}^{(J)})}\leq\varepsilon_{j}^{-\sigma/6},\ \
|Z_{J}|_{C^{3}(\Gamma_{j+1}^{(J)},2\TT^{d})}\leq\varepsilon_{j}^{-\sigma/3},
\end{split}
\end{equation}
\begin{equation}\label{20230813+}
\begin{split}
|A_{J}-B|_{C^{3}(\Gamma_{j}^{(J)})}\leq\varepsilon_{J}^{1/4},\ \
|Z_{J}|_{C^{3}(\Gamma_{0}^{(J)}),2\TT^{d}}\leq\varepsilon_{J}^{1/4}.
\end{split}
\end{equation}
%

$(3):$ $A_{J}$ has two eigenvalues $\me^{\pm i\alpha_{J}}$ with $\alpha_{J}\in\RR\cup i \RR.$ For
$\xi_{J}:=\mathrm{Re}\alpha_{J},$
\begin{equation}\label{202308140}
\begin{split}
|\xi_{J}-\xi|_{C^{3}(\Sigma_{j})}\leq\varepsilon_{J}^{1/4},0\leq j\leq J, \
|\xi_{J}-\rho|_{C^{3}(\Gamma_{0}^{(J)})}\leq\varepsilon_{J}^{1/4}.
\end{split}
\end{equation}
If $J\geq1,$ then
\begin{equation}\label{202308141}
\begin{split}
|\xi_{J}|_{\Gamma_{j+1}^{(J)}}\leq3\varepsilon_{j}^{1/4}/2,\ 0\leq j\leq J-1.
\end{split}
\end{equation}
Moreover, there is $k_{j}: \cup_{l=0}^{J}\Gamma_{l}^{(J)}\rightarrow\ZZ^{d},0\leq j\leq J-1,$ constant on each connected component of $\cup_{l=0}^{J}\Gamma_{l}^{(J)}$ with $0<|k_{j}|\leq N_{j}$ on $\Gamma_{j+1}^{(J)}$
and $k_{l}=0$ on $\Gamma_{j+1}^{(J)}$ for $l\geq j+1,$ such that
\begin{equation}\label{202308142}
\begin{split}
|\xi_{J}+\sum_{l=0}^{J-1}\langle k_{l}\rangle_{\omega}-\rho|_{\Gamma_{j+1}^{(J)}}\leq\varepsilon_{J}^{1/4}.
\end{split}
\end{equation}

$(4):$ $\cup_{l=0}^{J}\{\Gamma_{l}^{(J)}:|\sin\xi_{J}|\geq3\varepsilon_{J}^{1/20}/2\}$ has at most $|\ln\varepsilon_{0}|^{2J^{2}d}$ connected components, on
which $\xi_{J}$ is smooth and possesses the estimate
%
%
\begin{equation}\label{202310033}
\begin{split}
|\partial_{E}^{s}(\xi_{J}+\langle k_{J-1}\rangle_{\omega}-\xi_{J-1})|_{\Gamma_{j}^{(J)}}\leq \varepsilon_{J-1}^{1/4},j=0,\cdots,J, s=1,2,3,
\end{split}
\end{equation}
where $k_{J-1}$ is the one in $(3).$
Moreover,
\begin{equation}\label{202308144}
\begin{split}
\big|\rho(\{(\inf\Sigma,\sup\Sigma):|\sin\xi_{J}|\leq
3\varepsilon_{J}^{1/20}/2\})\big|\leq
\varepsilon_{J}^{1/24},
\end{split}
\end{equation}
and
\begin{equation}\label{202308145}
\begin{split}
\big|\xi_{J}(\Gamma_{j}^{J}\setminus\Sigma_{j})\big|\leq
\varepsilon_{J}^{7\sigma/8}, 0\leq j\leq J.
\end{split}
\end{equation}
\end{theorem}

$\mathbf{Sketch \ of\ the\ proof\
Theorem~\ref{mainresult}.}$ The proof of Theorem~\ref{mainresult} is based on the KAM theory of Eliasson \cite{Eliasson92} and Hadj Amor \cite{Amor09} for the reducibility
of the Schr\"odinger cocycle. Note that this theorem is the combination of Theorem 2.1 and Theorem 2.2 in \cite{zhao20} with the attached estimate \eqref{202310033}, which will play fundamental role in this work, we will just give the proof of this inequality.

Rewrite the Schr\"odinger cocycle $(\omega, A_{0}+F_{0})$ defined by \eqref{202402070} as $(\omega, A_{0}\me^{f_{0}}),$ where $
f_{0}(\cdot)=\left(
\begin{matrix}
0,\ \ 0
\\
-P(\cdot), 0
\end{matrix}
\right)$ and denote $\me^{\mi\xi_{0}}$ and $\me^{-\mi\xi_{0}}$ by the eigenvalues of $A_{0}.$ To construct the transformation, generally, the non-resonance condition is needed:
\begin{equation}\label{20231003+}
\begin{split}
|\xi_{0}-\langle k\rangle_{\omega}|\geq \varepsilon_{0}^{\sigma}|k|^{-\tau},\ 0\leq |k|\leq N_{0},
\end{split}
\end{equation}
where $\Gamma_{0}^{(1)}$ is the set of such $E's.$

For the $E$ such that \eqref{20231003+}
is fulfilled, we apply the classical construction of the KAM step
to get a near-identity transformation which conjugates $(\omega, A_{0}e^{f_{0}}),$ to $(\omega, A_{1}e^{f_{1}}),$ where $|f_{1}|$ is of $\varepsilon_{1}$ and $A_{1}\in SL(2,\RR)$ possesses eigenvalues $\me^{\pm \mi\xi_{1}}$ with
\begin{equation}\label{202401121}
\begin{split}
|\partial_{E}^{s}(\xi_{1}-\xi_{0})|_{\Gamma_{0}^{(1)}}\leq \varepsilon_{0}^{1/4},\ s=1,2,3.
\end{split}
\end{equation}

To describe the $E$ such that the equation \eqref{20231003+} will not hold,
we define
\begin{equation}\label{202310034}
\begin{split}
\mathcal{I}_{k}:=\{E:\ |\xi_{0}(E)-\langle k\rangle_{\omega}|< \varepsilon_{0}^{\sigma}|k|^{-\tau}\}.
\end{split}
\end{equation}
Since we assume $\omega\in DC_{\gamma,\tau},$ for the fixed $E,$ there exists only one $k\in\ZZ^{d}$ such that \eqref{20231003+} is violated, thus, the segment $\mathcal{I}_{k}$ is well-defined. The union of the segments $\mathcal{I}_{k}$ is the set $\Gamma_{1}^{(1)}.$

For these $E\in\mathcal{I}_{k_{0}},$ with $0<|k_{0}|\leq N_{0},$ we construct a time dependent matrix (in $su(1,1)$ topology)
\begin{equation}\label{202310035}
\begin{split}
H_{k_{0}}(\theta)=C_{A_{0}}Q_{k_{0}}(\theta)C_{A_{0}}^{-1}
\end{split}
\end{equation}
with
\begin{equation}\label{202310036}
\begin{split}
Q_{k_{0}}(\theta)=\left(
\begin{matrix}
e^{\mi2^{-1}\langle k_{0},\theta\rangle},0
\\
0,\ \ e^{-\mi2^{-1}\langle k_{0},\theta\rangle}
\end{matrix}
\right),\ C_{A_{0}}=\left(
\begin{matrix}
-1,-1
\\
e^{i\alpha_{0}}+E,\ \ e^{-i\alpha_{0}}+E
\end{matrix}
\right),
\end{split}
\end{equation}
conjugating $(\omega,A_{0}e^{f_{0}})$ to a new cocycle $(\omega,\tilde{A}_{0}e^{\tilde{f}_{0}}).$ Note that the matrix $C_{A_{0}}$ and $C^{-1}_{A_{0}}$ in \eqref{202310035} are constant ones, and will not change the eigenvalues, thus $\tilde{A}_{0}$ has
eigenvalues $\me^{\pm\mi \widetilde{\xi}_{0}}$ with
\begin{equation}\label{202310037}
\begin{split}
\widetilde{\xi}_{0}:=\xi_{0}-\langle k_{0}\rangle_{\omega},
\end{split}
\end{equation}
and the estimates
\begin{equation}\label{202310051}
\begin{split}
|\widetilde{\xi}_{0}|\leq \varepsilon_{0}^{\sigma}|k|^{-\tau}<\varepsilon_{0}^{\sigma}.
\end{split}
\end{equation}
The inequality above, together with the fact $\omega\in DC_{\gamma,\tau},$ yields,
\begin{equation*}
\begin{split}
|\widetilde{\xi}_{0}-\langle k\rangle_{\omega}|\geq \varepsilon_{0}^{\sigma}|k|^{-\tau},0\leq |k|\leq N_{0},
\end{split}
\end{equation*}
that is \eqref{20231003+} is fulfilled by $\widetilde{\xi}_{0}.$ Thus, with a near-identity change of variables, one can conjugate $(\omega,\tilde{A}_{0}\me^{\tilde{f}_{0}})$ to $(\omega,A_{1}\me^{f_{1}}),$
where $|f_{1}|$ is of $\varepsilon_{1}$ and $A_{1}\in SL(2,\RR)$ possesses eigenvalues $\me^{\pm \mi\alpha_{1}}.$ Set
\begin{equation}\label{202310050}
\begin{split}
\xi_{1}=\mathrm{Re}\alpha_{1},\ \
\rho_{1}=\xi_{1}+\langle k_{0}\rangle_{\omega},
\end{split}
\end{equation}
then we have
\begin{equation*}
\begin{split}
|\partial_{E}^{s}(\xi_{1}-\widetilde{\xi}_{0})|_{\Gamma_{1}^{(1)}}\leq \varepsilon_{0}^{1/4},\ s=1,2,3,
\end{split}
\end{equation*}
which, together with \eqref{202310037}, implies
\begin{equation}\label{202310038}
\begin{split}
|\partial_{E}^{s}(\xi_{1}-\xi_{0}+\langle k_{0}\rangle_{\omega})|_{\Gamma_{1}^{(1)}}\leq \varepsilon_{0}^{1/4},\ s=1,2,3.
\end{split}
\end{equation}

The estimates in \eqref{202401121} and \eqref{202310038} imply that the estimate
in \eqref{202310033} holds with $J=1.$

We proceed the induction. Assume that we have gotten the cocycle $(\omega, A_{J}), J=0,\cdots,\ell,$ with the related estimates and will construct the change of variables to transform $(\omega, A_{\ell})$ to $(\omega, A_{\ell+1})$ with the related estimates.
In this iteration, the non-resonance
condition is
\begin{equation}\label{202310039}
\begin{split}
|\xi_{\ell}-\langle k\rangle_{\omega}|\geq \varepsilon_{\ell}^{\sigma}|k|^{-\tau},
\ 0\leq |k|\leq N_{\ell}.
\end{split}
\end{equation}

We denote $\Gamma_{j}^{(\ell+1)} (0\leq j\leq \ell)$ by the set of $E\in\Gamma_{j}^{(\ell)}(0\leq j\leq \ell)$ such that
the equation \eqref{202310038} holds with $\ell$ and $\ell+1$ in place of $0$ and $1,$ respectively. Then, for $E\in\Gamma_{j}^{(\ell+1)}(0\leq j\leq \ell),$ one construct
a near-identity transformation conjugating $(\omega,A_{\ell}\me^{f_{\ell}})$ to
$(\omega,A_{\ell+1}\me^{f_{\ell+1}})$
where $|f_{\ell+1}|$ is of $\varepsilon_{\ell+1}$ and $A_{\ell+1}\in SL(2,\RR)$ possesses eigenvalues $\me^{\pm \mi\alpha_{\ell+1}}.$ Set $\xi_{\ell+1}=\mathrm{Re}\alpha_{\ell+1},$ then we have
\begin{equation}\label{202401160}
\begin{split}
|\partial_{E}^{s}(\xi_{\ell+1}-\xi_{\ell})|_{\Gamma_{j}^{(\ell+1)}}\leq \varepsilon_{\ell}^{1/4},\ s=1,2,3,0\leq j\leq \ell.
\end{split}
\end{equation}

Set
\begin{equation}\label{20231004+}
\begin{split}
\Gamma_{\ell+1}^{(\ell+1)}:=\cup_{0\leq j\leq\ell}\big(\Gamma_{j}^{(\ell)}\setminus\Gamma_{j}^{(\ell+1)}\big).
\end{split}
\end{equation}
Thus, for $E\in\Gamma_{\ell+1}^{(\ell+1)},$ we know that the equation \eqref{202310039} is
violated. In this case, we  proceed the same operations as the ones in the case $E\in\Gamma_{1}^{(1)}.$ Moreover, with the similar discussions above, we also prove that the estimate in \eqref{202310038} holds with the indexes $0$ and $1$ replaced by $\ell$ and $\ell+1,$ respectively.

The discussions above imply that the estimate in \eqref{202310033} holds with $J=\ell+1,$ therefore for $J\in\NN.$

Now, we have constructed a family of sets $\{\Gamma_{j}^{(J)}\}_{j=0,\cdots,J,J\geq1}$ and set
\begin{equation}\label{202310040}
\begin{split}
\Sigma_{j}:=\bigcap_{J\geq j}\Gamma_{j}^{(J)}\setminus
\bigcup_{k\in\ZZ^{d}}\rho^{-1}(\langle k\rangle_{\omega}).
\end{split}
\end{equation}

Up to here, we have finished the proof of Theorem~\ref{mainresult}.

Set
\begin{equation}\label{202308146+}
\begin{split}
\rho_{J}:=\xi_{J}+\sum_{l=0}^{J-1}\langle k\rangle_{\omega}
\end{split}
\end{equation}
with $\rho_{0}=\xi_{0}$. Then $\{\rho_{J}\}_{J\in\ZZ}$ is an approximation of $\rho$
with
\begin{equation}\label{202308146}
\begin{split}
|\rho_{J}-\rho|_{\Sigma_{j}}\leq
\varepsilon_{J}^{1/4}, \forall J\geq 0.
\end{split}
\end{equation}
The following estimates will play similar role that (19) does in \cite{zhao20}.
\begin{lemma}\label{resonantcase}
Set $J\in\NN,$ and $(E_{*},E_{**})\subset\Gamma_{j}^{(J)},0\leq j\leq J,$ be any connected component. Assume $E\in(E_{*},E_{**}),$ then for $\rho_{J}$ defined by \eqref{202308146+}, we have
\begin{equation}\label{202401120}
\begin{split}
\rho'_{0}
=\frac{1}{2\sin\xi_{0}},\ \rho''_{0}
=\frac{-\cos\xi_{0}}{4\sin^{3}\xi_{0}},\ \rho'''_{0}
=\frac{1+2\cos^{2}\xi_{0}}{8\sin^{5}\xi_{0}},
\end{split}
\end{equation}
and, for $j=0,1,\cdots,J, J\geq1,$
\begin{equation}\label{202309096}
\begin{split}
|\rho'_{J}-(2\sin\xi_{0})^{-1}|_{\Gamma_{j}^{(J)}}
&=O_{1}(\varepsilon_{j}^{1/4}),\\
|\rho_{J}'^{-1}\rho_{J}''+\frac{\cos\xi_{0}}{2\sin^{2}\xi_{0}}|_{\Gamma_{j}^{(J)}}
&=
\frac{2\sin^{2}\xi_{0}O_{2}(\varepsilon_{0}^{1/4})
+\cos\xi_{0}O_{1}(\varepsilon_{0}^{1/4})
}
{\sin\xi_{0}(1+O_{1}(\varepsilon_{0}^{1/4}))}
,\\
|\rho_{J}'^{-1}\rho_{J}'''-
\frac{1+2\cos^{2}\xi_{0}}{4\sin^{4}\xi_{0}}|_{\Gamma_{j}^{(J)}}
&=\frac{4\sin^{4}\xi_{0}O_{3}(\varepsilon_{0}^{1/4})
-O_{1}(\varepsilon_{0}^{1/4})(1+2\cos^{2}\xi_{0})
}
{2\sin^{3}\xi_{0}(1+O_{1}(\varepsilon_{0}^{1/4}))}.
\end{split}
\end{equation}
\end{lemma}

\begin{proof}
We prove the estimates in \eqref{202401120} first.
Set $
A_{0}=\left(
\begin{matrix}
-E,-1
\\
1,\ \ 0
\end{matrix}
\right)$ with eigenvalues $\me^{\mi\xi_{0}}$ and $\me^{-\mi\xi_{0}}.$ Then
\begin{equation*}
\begin{split}
tr A_{0}=-E=2\cos\xi_{0},
\end{split}
\end{equation*}
which, together with $\rho_{0}=\xi_{0},$ yields the estimates in \eqref{202401120}.

Now, we consider the estimates in \eqref{202309096}. Note that
\begin{equation}\label{202310054}
\begin{split}
\rho_{1}=\xi_{1}+\langle k_{0}\rangle_{\omega}
=\xi_{0}+\widetilde{\rho}_{1}(\xi_{0}),
\end{split}
\end{equation}
where
\begin{equation*}
\begin{split}
\widetilde{\rho}_{1}(E)=\xi_{1}+\langle k_{0}\rangle_{\omega}-\xi_{0}.
\end{split}
\end{equation*}
By applying \eqref{202310033} with $J=1$ we get
\begin{equation*}
\begin{split}
|\partial_{E}^{s}\widetilde{\rho}_{1}(E)|_{\Gamma_{j}^{(1)}}
=O_{s}(\varepsilon_{\ell}^{1/4}),\ s=1,2,3,j=0,1.
\end{split}
\end{equation*}
The estimates above, together with the estimates in \eqref{202401120} and \eqref{202310054}, yield
\begin{equation*}
\begin{split}
\big|\rho'_{1}-(2\sin\xi_{0})^{-1}\big|_{\Gamma_{j}^{(1)}}
&=O_{1}(\varepsilon_{0}^{1/4}),j=0,1,\\
\big|\rho''_{1}+\frac{\cos\xi_{0}}{4\sin^{3}\xi_{0}}\big|_{\Gamma_{j}^{(1)}}
&=O_{2}(\varepsilon_{0}^{1/4}),j=0,1,\\
\big|\rho'''_{1}-\frac{1+2\cos^{2}\xi_{0}}{8\sin^{5}\xi_{0}}\big|_{\Gamma_{j}^{(1)}}
&=O_{3}(\varepsilon_{0}^{1/4}),j=0,1.
\end{split}
\end{equation*}
Thus, we have
\begin{equation*}
\begin{split}
\big|\rho'^{-1}_{1}\rho''_{1}+\frac{\cos\xi_{0}}{2\sin^{2}\xi_{0}}\big|_{\Gamma_{j}^{(1)}}
=
\frac{2\sin^{2}\xi_{0}O_{2}(\varepsilon_{0}^{1/4})
+\cos\xi_{0}O_{1}(\varepsilon_{0}^{1/4})
}
{\sin\xi_{0}(1+O_{1}(\varepsilon_{0}^{1/4}))},
\end{split}
\end{equation*}
and
\begin{equation*}
\begin{split}
\big|\rho'^{-1}_{1}\rho'''_{1}-\frac{1+2\cos^{2}\xi_{0}}{4\sin^{4}\xi_{0}}
\big|_{\Gamma_{j}^{(1)}}
&=\frac{4\sin^{4}\xi_{0}O_{3}(\varepsilon_{0}^{1/4})
-O_{1}(\varepsilon_{0}^{1/4})(1+2\cos^{2}\xi_{0})
}
{2\sin^{3}\xi_{0}(1+O_{1}(\varepsilon_{0}^{1/4}))}.
\end{split}
\end{equation*}
That is the estimates in \eqref{202309096} hold with $J=1.$

Now, assume that we have proved the cases $J=1,2,\cdots,\ell, \ell\geq1,$ and then we consider the case $J=\ell+1.$ Note that
\begin{equation}\label{202401114}
\begin{split}
\rho_{\ell+1}=\xi_{\ell}+\sum_{j=0}^{\ell}\langle k_{j}\rangle_{\omega}
=\rho_{\ell}+\widetilde{\rho}_{\ell+1}(\xi_{\ell}),
\end{split}
\end{equation}
where
\begin{equation*}
\begin{split}
\widetilde{\rho}_{\ell+1}(E)=\xi_{\ell+1}+\langle k_{\ell}\rangle_{\omega}-\xi_{\ell},
\end{split}
\end{equation*}
and by applying \eqref{202310033} with $J=\ell$ we get
\begin{equation*}
\begin{split}
|\partial_{E}^{s}\widetilde{\rho}_{\ell+1}(E)|_{\Gamma_{j}^{(\ell+1)}}
=O_{s}(\varepsilon_{\ell}^{1/4}),\ s=1,2,3,j=0,1,\cdots,\ell+1,
\end{split}
\end{equation*}
which, together with \eqref{202401114} and the estimates in \eqref{202309096} with $J=\ell,$ imply
\begin{equation*}
\begin{split}
\big|\rho'_{\ell+1}-(2\sin\xi_{0})^{-1}\big|_{\Gamma_{j}^{(\ell+1)}}
&=\sum_{j=0}^{\ell}O_{1}(\varepsilon_{j}^{1/4}),\\
\big|\rho''_{\ell+1}+\frac{\cos\xi_{0}}{4\sin^{3}\xi_{0}}\big|_{\Gamma_{j}^{(\ell+1)}}
&=\sum_{j=0}^{\ell}O_{2}(\varepsilon_{j}^{1/4}),\\
\big|\rho'''_{\ell+1}-\frac{1+2\cos^{2}\xi_{0}}{8\sin^{5}\xi_{0}}
\big|_{\Gamma_{j}^{(\ell+1)}}
&=\sum_{j=0}^{\ell}O_{3}(\varepsilon_{j}^{1/4}).
\end{split}
\end{equation*}
The estimates above yield the estimates in \eqref{202309096} with $J=\ell+1.$
Therefore, for all $J\in\RR.$
\end{proof}
\subsection{Spectral transform}
For $E\in\Sigma,$ let $\mathcal{K}(E)$ and $\mathcal{J}(E)$ be two linearly independent generalized eigenvectors of $H_{\theta}$ and consider the spectral transform $\mathcal{S}(q)$ defined as follows: for any $q\in\ell^{2},$ set
\begin{equation}\label{202308147}
\begin{split}
S(q)(t):=\left(
\begin{matrix}
\sum_{n}q_{n}(t)\mathcal{K}_{n}(E)\\
\\
\sum_{n}q_{n}(t)\mathcal{J}_{n}(E)
\end{matrix}
\right).
\end{split}
\end{equation}
Given any matrix of measures on $\RR,$ for example $d\psi:=\left(
\begin{matrix}
d\psi_{11},d\psi_{12}
\\
d\psi_{21},d\psi_{22}
\end{matrix}
\right),$
let $L^{2}(d\psi)$ be the space of the vectors $G=(g_{j})_{j=1,2},$ with $g_{j}$ functions of $E\in\RR$ satisfying
\begin{equation}\label{202308148}
\begin{split}
\|G\|_{L^{2}(d\psi)}^{2}:=\sum_{j,k=1,2}\int_{\RR}
g_{j}\overline{g}_{k}d\psi_{jk}<\infty.
\end{split}
\end{equation}

\begin{proposition}\label{proposition}
Set $\widetilde{\Sigma}:=\cup_{j\geq0}\Sigma_{j},$ which is the full measure subset of the spectrum. Then for any fixed $\theta\in\TT^{d},$ any $E\in\widetilde{\Sigma},$
there exist two linearly independent generalized eigenvectors $\mathcal{K}(E)$
and $\mathcal{J}(E)$ of $H_{\theta}$ with the following properties: Define the spectral transform according to \eqref{202308147} and consider the matrix of measures $d\psi$ given by
\begin{equation}\label{202308149}
\begin{split}
d\psi|_{\Sigma}=\frac{1}{\pi}\left(
\begin{matrix}
\rho',0
\\
0,\rho'
\end{matrix}
\right)dE,\ d\psi|_{\RR\setminus\Sigma}=0,
\end{split}
\end{equation}
we have, for $q\in\ell^{2},$
\begin{equation}\label{202308120}
\begin{split}
(1-\varepsilon_{0}^{\sigma^{2}/10})\|q\|_{\ell^{2}}^{2}\leq
\|S(q)\|_{L^{2}(d\varphi)}^{2}
\leq(1+\varepsilon_{0}^{\sigma^{2}/10})\|q\|_{\ell^{2}}^{2},
\end{split}
\end{equation}
and
\begin{equation}\label{202308121}
\begin{split}
|\frac{1}{\pi}\int_{\Sigma}\big\{g_{1}(E)\mathcal{K}_{n}(E)
+g_{2}(E)\mathcal{J}_{n}(E)\big\}\rho'dE-q_{n}|
\leq\varepsilon_{0}^{\sigma^{2}/10}\|q\|_{\ell^{\infty}}.
\end{split}
\end{equation}
Furthermore, the functions $\mathcal{K}_{n}(E)$ and $\mathcal{J}_{n}(E)$ have the following properties:
\begin{equation}\label{202308122}
\begin{split}
\mathcal{K}_{n}(E)&=\sum_{n_{*}=n,n\pm1}\beta_{n,n_{*}}(E)\sin n_{*}\rho(E),\\
\mathcal{J}_{n}(E)&=\sum_{n_{*}=n,n\pm1}\beta_{n,n_{*}}(E)\cos n_{*}\rho(E)
\end{split}
\end{equation}
with $\rho=\rho(A_{0}+F_{0})$ and
\begin{equation}\label{202308123}
\begin{split}
|\beta_{n,n_{*}}-\delta_{n,n_{*}}|_{\Sigma_{0}}\leq\varepsilon_{0}^{1/4},
|\beta_{n,n_{*}}|_{\Sigma_{j+1}}\leq\varepsilon_{j}^{\sigma},j\geq0.
\end{split}
\end{equation}

Given any $J\in\NN,$ there exist $\beta^{J}_{n,n_{*}},$ smooth on each connected component of $\Gamma_{j}^{(J)},$  satisfying
\begin{equation}\label{202308124}
\begin{split}
|\beta_{n,n_{*}}^{J}-\delta_{n,n_{*}}|_{C^{2}(\Gamma_{0}^{J})}
\leq\varepsilon_{0}^{1/4},
\end{split}
\end{equation}
and if $J\geq1,$
\begin{equation}\label{202308125}
\begin{split}
|\beta_{n,n_{*}}^{J}|_{C^{1}(\Gamma_{j+1}^{J})}\leq\varepsilon_{j}^{3\sigma},0\leq j\leq J-1.
\end{split}
\end{equation}
Moreover,
\begin{equation}\label{202308126}
\begin{split}
|\beta_{n,n_{*}}^{J}-\beta_{n,n_{*}}|_{\Sigma_{j}}
\leq10\varepsilon_{J}^{1/4},0\leq j\leq J.
\end{split}
\end{equation}
\end{proposition}

$\mathbf{Idea\ of \ the \ proof.}$ We will follow the proof idea of Proposition~2.4 in \cite{zhao20}.
We construct Bloch-waves of Schr\"odinger operator $H_{\theta}$ on $\widetilde{\Sigma}.$ The \eqref{202308131} shows that there exist $Z:\widetilde{\Sigma}\times 2\TT^{d}\rightarrow SL(2,C)$ and $B:\widetilde{\Sigma}\rightarrow SL(2,C)$ with two
eigenvalues $\me^{\mi\rho},$ such that
\begin{equation*}
\begin{split}
Z(\cdot+\omega)^{-1}(A_{0}+F_{0}(\cdot))Z(\cdot)=B.
\end{split}
\end{equation*}
Set
\begin{equation*}
\begin{split}
Z=\left(
\begin{matrix}
Z_{11},Z_{12}
\\
Z_{21},Z_{22}
\end{matrix}
\right),\ B=\left(
\begin{matrix}
B_{11},B_{12}
\\
B_{21},B_{22}
\end{matrix}
\right),
\end{split}
\end{equation*}
and
\begin{equation*}
\begin{split}
\widetilde{f}_{n}(\theta):&=[Z_{11}(\theta+(n-1)\omega)B_{12}
-Z_{12}(\theta+(n-1)\omega)B_{11}]
\me^{-\mi \rho}\\
&+Z_{12}(\theta+(n-1)\omega),\\
\widetilde{\psi}_{n}:&=\me^{-\mi n\rho}\widetilde{f}_{n}(\theta).
\end{split}
\end{equation*}
Then the sequence $\widetilde{\psi}:=\{\widetilde{\psi}_{n}\}_{n\in\ZZ}$ satisfies
$H_{\theta}\widetilde{\psi}=E\widetilde{\psi},\forall E\in\widetilde{\Sigma}.$
Set
\begin{equation*}
\begin{split}
\psi_{n}:=\me^{\mi  n\rho}f_{n}(\theta),\ \text{with}\
f_{n}(\theta)=\left\{
\begin{matrix}
\widetilde{f}_{n}(\theta),\ E\in\Sigma_{0},
\\
\widetilde{f}_{n}(\theta)\sin^{5}\xi,E\in\Sigma_{j+1},j\geq0.
\end{matrix}
\right.
\end{split}
\end{equation*}
Define $\mathcal{K}_{n}(E):=\mathrm{Im}(\me^{\mi  n\rho}f_{n}\overline{f}_{0})$ and $\mathcal{J}_{n}(E):=\mathrm{Re}(\me^{\mi  n\rho}f_{n}\overline{f}_{0})$ on $\widetilde{\Sigma}$ and
$\mathcal{K}_{n}|_{\RR\setminus\widetilde{\Sigma}}=
\mathcal{J}_{n}|_{\RR\setminus\widetilde{\Sigma}}=0.$ Moreover, note that
\begin{equation*}
\begin{split}
\me^{\mi n\rho}f_{n}\overline{f}_{0}=\sum_{n_{*}=n,n\pm1}\beta_{n,n_{*}}\me^{\mi n_{*}\rho},
\end{split}
\end{equation*}
where $\beta_{n,n_{*}}$ will be fixed later with certain estimate. Then
\begin{equation*}
\begin{split}
\mathcal{K}_{n}(E)&=\sum_{n_{*}=n,n\pm1}\beta_{n,n_{*}}\sin(n_{*}\rho(E)),\\
\mathcal{J}_{n}(E)&=\sum_{n_{*}=n,n\pm1}\beta_{n,n_{*}}\cos(n_{*}\rho(E)).
\end{split}
\end{equation*}
Then, we get the estimates in \eqref{202308123}-\eqref{202308126} and
\begin{equation*}
\begin{split}
\big|\frac{1}{\pi}\int_{\Sigma}(\mathcal{K}_{n}^{2}+\mathcal{J}_{n}^{2})\rho'dE-1
\big|\leq\varepsilon_{0}^{\sigma^{2}/8},
\end{split}
\end{equation*}
\begin{equation*}
\begin{split}
\big|\frac{1}{\pi}\int_{\Sigma}(\mathcal{K}_{m}\mathcal{K}_{n}
+\mathcal{J}_{m}\mathcal{J}_{n})\rho'dE
\big|\leq\varepsilon_{0}^{\sigma^{2}/8}|n-m|^{-1-\sigma/6}, m\neq n.
\end{split}
\end{equation*}
With the two inequalities above and the same trick in \cite{Zhao16}, we get the estimates in \eqref{202308120} and \eqref {202308121}.
\section{Oscillatory integral on the spectrum}
In Theorem~\ref{mainresult}, we have constructed the division of the interval $[\inf\Sigma,\sup\Sigma].$ Based on this division, we will give the estimate
of an integral on the spectrum, which will be applied in analyzing the time evolution, and deducing dispersion in the next section.
To simplify the notations, we set
\begin{equation}\label{202308116}
\begin{split}
\mathcal{I}_{M}(S)=\int_{S}h\me^{\mi t\sqrt{E+3}}\cos M\rho\cdot\rho'dE,\ S\subset\Sigma,
\end{split}
\end{equation}
and
\begin{equation}\label{202308117}
\begin{split}
\mathcal{I}_{M}^{J}(S)=\int_{S}h_{J}\me^{\mi t\sqrt{E+3}}\cos M\rho\cdot\rho'dE,\ S\subset\Sigma.
\end{split}
\end{equation}
\begin{lemma}\label{mainlemma}
Assume that $h$ maps $\widetilde{\Sigma}$ to $\RR$ such that for any $J\geq0,$ there exists a function $h_{J}$ mapping $\cup_{0\leq j\leq J}\Gamma_{j}^{(J)}$ to $\RR,$
which is $C^{3}$ on each connected component of $\Gamma_{j}^{(J)}$ and satisfies the following estimates:
\begin{equation}\label{202308113}
\begin{split}
|h_{J}-h|_{\Sigma_{j}}\leq10\varepsilon_{J}^{1/4},\ 0\leq j\leq J,
|h_{J}-h|_{\Sigma_{j}}\leq10\varepsilon_{J}^{1/4},\ j\geq J+1,
\end{split}
\end{equation}
\begin{equation}\label{202308114}
\begin{split}
|h_{J}|_{C^{1}(\Gamma_{0}^{(J)})}\leq16/15,
\end{split}
\end{equation}
and
\begin{equation}\label{202308115}
\begin{split}
|h_{J}|_{C^{1}(\Gamma_{j+1}^{(J)})}\leq10\varepsilon_{j}^{3\sigma},\ 0\leq j\leq J-1, J\geq1.
\end{split}
\end{equation}
Then, there exists a positive constant $a,$ such that for any
$M\in\RR,$ one has
\begin{equation}\label{202308115+}
\begin{split}
|\mathcal{I}_{M}(\Sigma)|\leq7624\langle t\rangle^{-1/3}|\ln\varepsilon_{0}|^{ad(\ln\ln(2+\langle t\rangle))^{2}}.
\end{split}
\end{equation}
\end{lemma}

\subsection{Auxiliary lemmas}
To prove Lemma~\ref{mainlemma}, we give some auxiliary lemmas to transform the estimate about $\mathcal{I}_{M}(\Sigma)$ to the ones about a new integral.
\begin{lemma}\label{auxiliarylemma1}
For any positive $J$ and any $M\in\RR,$ under the assumptions of Lemma~\ref{mainlemma}, one has
\begin{equation}\label{202308118}
\begin{split}
|\mathcal{I}_{M}(\Sigma)-\mathcal{I}_{M}^{J}(\Sigma)|
\leq\varepsilon_{J}^{3\sigma/4}.
\end{split}
\end{equation}
\end{lemma}

\begin{lemma}\label{auxiliarylemma2}
Assume that for some positive $J\geq0$ the function $h_{J}$ satisfies the estimates \eqref{202308114} and \eqref{202308115},
then for every $M\in\RR\setminus\{0\}$ and $t\in\RR,$ one has
\begin{equation}\label{202308119}
\begin{split}
|\mathcal{I}_{M}^{J}(\Sigma)|
\leq32\big\{|\ln\varepsilon_{0}|^{2J^{2}d}
+(\sup\Sigma-\inf\Sigma)\langle t\rangle\big\}(15|M|)^{-1}.
\end{split}
\end{equation}
\end{lemma}

The proof of Lemmas~\ref{auxiliarylemma1},~\ref{auxiliarylemma2} are similar with the ones of Lemmas~3.2,3.3 in \cite{zhao20}, we omit the details.

Note that the hypothesis about $M$ in Lemma~\ref{auxiliarylemma2} is $M\in\RR\setminus\{0\},$ which corresponds to $M's$ that are large enough. While the discussions about $M=0$ (without loss of generality, we set $M\in\RR$) is much complicated, which involves the Van der Corput lemma. In this case, we estimate the derivatives of $\partial_{\rho_{J}}^{j}\sqrt{E+3}, j=2,3,$ first.
\begin{lemma}\label{gamma0estimate}
Let $(E_{*},E_{**})\subset\cup_{l=0}^{J}\{\Gamma_{l}^{(J)}
:|\sin\xi_{J}|\geq3\varepsilon_{J}^{1/20}/2\}$ be any connected component, then for any $E\in(E_{*},E_{**}),$ we have
\begin{equation}\label{202308159}
\begin{split}
\Big|\frac{d^{2}\sqrt{E+3}}{d\rho_{J}^{2}}\Big|\geq\frac{1}{200},\ -1\leq\cos\xi_{0}\leq\frac{1}{3} \text{or}\ \frac{3}{5}\leq\cos\xi_{0}\leq1,
\end{split}
\end{equation}
and
\begin{equation}\label{+202308159}
\begin{split}
 \Big|\frac{d^{3}\sqrt{E+3}}{d\rho_{J}^{3}}\Big|\geq\frac{1}{200},\
 \frac{1}{3}<\cos\xi_{0}<\frac{3}{5}.
\end{split}
\end{equation}
\end{lemma}

\begin{proof}
The fact we focus on the $E's$ such that $|\sin\xi_{J}|\geq3\varepsilon_{J}^{1/20}/2$ implies that there exists a neighborhood of $\langle k\rangle_{\omega}, \forall k\in\ZZ^{d},$ which is separated from the connected component
$(E_{*},E_{**}).$ Thus, $\rho_{J}(E)$ is increasing on $(E_{*},E_{**})$ and the inverse function $E(\rho_{J})$ is well-defined.

The chain rule yields, on $\rho_{J}((E_{*},E_{**})),$
\begin{equation}\label{202309098}
\begin{split}
\frac{\partial E}{\partial\rho_{J}}&=(\rho_{J}')^{-1},\ \
\frac{\partial^{2} E}{\partial\rho_{J}^{2}}=-(\rho_{J}')^{-3}\rho_{J}'',\\
\frac{\partial^{3} E}{\partial\rho_{J}^{3}}&=
(\rho_{J}')^{-3}\big\{3(\rho_{J}'^{-1}\rho_{J}'')^{2}-
\rho_{J}'^{-1}\rho_{J}'''\big\},
\end{split}
\end{equation}
which, together with traditional calculations, shows
\begin{equation}\label{20230821+}
\begin{split}
\frac{\partial\sqrt{E+3}}{\partial\rho_{J}}&=
\frac{1}{2}(E+3)^{\frac{-1}{2}}\frac{dE}{d\rho_{J}},\\
\frac{\partial^{2}\sqrt{E+3}}{\partial\rho_{J}^{2}}
&=-\frac{1}{4}(E+3)^{\frac{-1}{2}}\big\{(E+3)^{-1}
\big(\frac{dE}{d\rho_{J}}\big)^{2}
-2\frac{d^{2}E}{d\rho_{J}^{2}}\big\}\\
&=-\frac{1}{4}(E+3)^{\frac{-1}{2}}\rho_{J}'^{-2}\big\{(E+3)^{-1}
+2\rho_{J}'^{-1}\rho_{J}''\big\},
\end{split}
\end{equation}
and
\begin{equation}\label{202308210}
\begin{split}
\frac{\partial^{3}\sqrt{E+3}}{\partial\rho_{J}^{3}}
&=\frac{3}{8}(E+3)^{\frac{-5}{2}}\big(\frac{dE}{d\rho_{J}}\big)^{3}
-\frac{3}{4}(E+3)^{\frac{-3}{2}}\frac{dE}{d\rho_{J}}
\frac{d^{2}E}{d\rho_{J}^{2}}
\\
&
+\frac{1}{2}(E+3)^{\frac{-1}{2}}\frac{d^{3}E}{d\rho_{J}^{3}}\\
&=\{8(E+3)^{1/2}\rho_{J}'^{3}\}^{-1}\big\{3(E+3)^{-2}
+6(E+3)^{-1}\rho_{J}'^{-1}\rho_{J}''\\
&+12(\rho_{J}'^{-1}\rho_{J}'')^{2}
-4\rho_{J}'^{-1}\rho_{J}'''\}.
\end{split}
\end{equation}
$\mathbf{Case \ one}:$ $-1\leq\cos\xi_{0}\leq\frac{1}{3}$ or $\frac{3}{5}\leq\cos\xi_{0}\leq1.$
The estimates in \eqref{202309096}, together with the fact $E=-2\cos\xi_{0},$ imply
\begin{equation*}
\begin{split}
(E+3)^{-1}&+2\rho_{J}'^{-1}\rho_{J}''
=
\frac{1-3\cos\xi_{0}+\cos^{2}\xi_{0}}{(3-2\cos\xi_{0})\sin^{2}\xi_{0}}\\
&+
\frac{O_{1}(\varepsilon_{0}^{1/4})2\cos\xi_{0}
+4\sin^{2}\xi_{0}O_{2}(\varepsilon_{0}^{1/4})}
{\sin\xi_{0}(1+O_{1}(\varepsilon_{0}^{1/4}))},
\end{split}
\end{equation*}
which, together with \eqref{20230821+}, yields
\begin{equation*}
\begin{split}
\Big|\frac{\partial^{2}\sqrt{E+3}}{\partial\rho_{J}^{2}}\Big|
&=\frac{4\sin^{2}\xi_{0}\big|(E+3)^{-1}+2\rho_{J}'^{-1}\rho_{J}''\big|}
{4(3-2\cos\xi_{0})^{1/2}(1+O_{1}(\varepsilon_{0}^{1/4}))^{2}}
\\
&>\frac{\sin^{2}\xi_{0}}{3}\big\{\frac{1}{50\sin^{2}\xi_{0}}
+\sin^{-1}\xi_{0}O(\varepsilon_{0}^{1/4})\big\}>\frac{1}{200}.
\end{split}
\end{equation*}

$\mathbf{Case \ two}:$ $\frac{1}{3}<\cos\xi_{0}<\frac{3}{5}.$ In this case we have $\frac{4}{5}<|\sin\xi_{0}|<\frac{2\sqrt{2}}{3}.$ Then \eqref{202309096} and \eqref{202308210} yield
\begin{equation*}
\begin{split}
\Big|\frac{\partial^{3}\sqrt{E+3}}{\partial\rho_{J}^{3}}\Big|
&=\{8(E+3)^{1/2}|\rho_{J}'|^{3}\}^{-1}
\Big|\frac{3}{(3-2\cos\xi_{0})^{2}}
-\frac{3\cos\xi_{0}}{(3-2\cos\xi_{0})\sin^{2}\xi_{0}}\\
&+\frac{3\cos^{2}\xi_{0}}{\sin^{4}\xi_{0}}
-\frac{1+2\cos^{2}\xi_{0}}{\sin^{4}\xi_{0}}
+\frac{O(\varepsilon_{0}^{1/4})}{|\sin\xi_{0}|^{3}}\Big|\\
&=\{8(E+3)^{1/2}|\rho_{J}'|^{3}\}^{-1}
\big\{\frac{|-6+3\cos\xi_{0}-\cos^{2}\xi_{0}|}{\sin^{2}\xi_{0}(3-2\cos\xi_{0})^{2}}
+\frac{O(\varepsilon_{0}^{1/4})}{|\sin\xi_{0}|^{3}}\big\}\\
&>\frac{1}{200}.
\end{split}
\end{equation*}
%

\end{proof}

\subsection{The estimate of integral with general M}
With the estimates in Lemma~\ref{gamma0estimate}, by applying the Van der Corput lemma, we are ready to estimate $\mathcal{I}_{M}^{J}(\Sigma)$ with $M\in\RR.$
\begin{lemma}\label{auxiliarylemma3}
Assume that for some positive $J\geq0$ the function $h_{J}$ satisfies the estimates \eqref{202308114} and \eqref{202308115},
then for every $M\in\RR$ and $t\in\RR,$ one has
\begin{equation}\label{20230815+}
\begin{split}
|\mathcal{I}_{M}^{J}(\Sigma)|
\leq|\ln\varepsilon_{0}|^{2J^{2}d}7622\langle t\rangle^{-1/3}
+\varepsilon_{J}^{3\sigma/4}/2+2|M|\varepsilon_{J}^{1/4}.
\end{split}
\end{equation}
\end{lemma}
\begin{proof}
 Set
\begin{equation}\label{202308150}
\begin{split}
\mathcal{R}_{M}^{J}:=\sum_{j=0}^{J}\int_{\{\Gamma_{j}^{(J)}
:|\sin\xi|>\varepsilon_{J}^{1/20}\}}h_{J}\me^{\mi t\sqrt{E+3}}
\cos M\rho_{J}\cdot\rho_{J}'dE
\end{split}
\end{equation}
as the approximation of $\mathcal{I}_{M}^{J}(\Sigma)$ defined by \eqref{202308117}. Thus,
\begin{equation}\label{202308151}
\begin{split}
\mathcal{I}_{M}^{J}(\Sigma)&-\mathcal{R}_{M}^{J}\\
&=-\int_{\cup_{0\leq j\leq J}\{\Gamma_{j}^{(J)}\setminus\Sigma_{j}
:|\sin\xi|>\varepsilon_{J}^{1/20}\}}h_{J}\me^{\mi t\sqrt{E+3}}
\cos M\rho_{J}\cdot\rho_{J}'dE
\end{split}
\end{equation}
\begin{equation}\label{202308152}
\begin{split}
+\int_{\cup_{0\leq j\leq J}\{\Sigma_{j}
:|\sin\xi|>\varepsilon_{J}^{1/20}\}}h_{J}\me^{\mi t\sqrt{E+3}}
(\cos M\rho\cdot\rho'-\cos M\rho_{J}\cdot\rho_{J}')dE
\end{split}
\end{equation}
\begin{equation}\label{202308153}
\begin{split}
+\int_{\cup_{0\leq j\leq J}\{\Sigma_{j}
:|\sin\xi|\leq\varepsilon_{J}^{1/20}\}}h_{J}\me^{\mi t\sqrt{E+3}}
\cos M\rho\cdot\rho'dE
\end{split}
\end{equation}
\begin{equation}\label{202308154}
\begin{split}
+\int_{\cup_{j\geq J+1}\Sigma_{j}}h_{J}\me^{\mi t\sqrt{E+3}}
\cos M\rho\cdot\rho'dE.
\end{split}
\end{equation}

The estimate $|\rho(\Sigma_{j+1})|\leq 3|\ln\varepsilon_{j}|^{2d}\varepsilon_{j}^{\sigma}$ ensures that the term in \eqref{202308154} is bounded by
\begin{equation*}
\begin{split}
\sum_{j=J+1}^{\infty}48|\ln\varepsilon_{j}|^{2d}\varepsilon_{j}^{\sigma}/5\leq
\varepsilon_{J}^{3\sigma/4}/8.
\end{split}
\end{equation*}

To consider the term in \eqref{202308153}, we fix our attention on the $E\in\Sigma_{j}, 0\leq j\leq J$ such that $|\sin\xi|\leq\varepsilon_{J}^{1/20}.$
The first inequality in \eqref{202308140} implies
\begin{equation*}
\begin{split}
|\sin\xi_{J}|\leq3\varepsilon_{J}^{1/4}/2,
\end{split}
\end{equation*}
which, together with \eqref{202308144} and \eqref{202308114}, implies that this term is bounded by
\begin{equation*}
\begin{split}
16\varepsilon_{J}^{1/24}/15\leq\varepsilon_{J}^{5\sigma}.
\end{split}
\end{equation*}

As for the term in \eqref{202308151}, the fact given in \eqref{202308145}, together with \eqref{202308114}, implies this term is bounded by
\begin{equation*}
\begin{split}
16(J+1)\varepsilon_{J}^{7\sigma/8}/15\leq\varepsilon_{J}^{3\sigma/4}/8.
\end{split}
\end{equation*}

Now, we consider the term defined by \eqref{202308152}.
On $\cup_{0\leq j\leq J}\{\Sigma_{j}: |\sin\xi|\geq\varepsilon_{J}^{1/20}\},$
by the first inequality in \eqref{202308140} one has
\begin{equation*}
\begin{split}
|\sin\xi_{J}|\geq\varepsilon_{J}^{1/20}/2.
\end{split}
\end{equation*}
%
The inequality above, together with the inequalities in \eqref{202308137}-\eqref{20230813+}, yields
\begin{equation*}
\begin{split}
|\rho'-\rho'_{J}|&=\frac{1}{2}\big|\sin^{-1}\xi(tr B)'-\sin^{-1}\xi_{J}(tr A_{J})'\big|\\
&\leq\frac{|(tr A_{J})'||\sin\xi-\sin\xi_{J}|+|\sin\xi_{J}||(tr B)'-(tr A_{J})'|}
{2\sin\xi_{J}\sin\xi}\\
&\leq2\varepsilon_{J}^{-10}(2\varepsilon_{J}^{1/4}N_{J}^{10\tau}
+2\varepsilon_{J}^{1/4})\leq\varepsilon_{J}^{1/10}.
\end{split}
\end{equation*}
Moreover, with the estimate in \eqref{202308146} we also have
\begin{equation*}
\begin{split}
|\cos M\rho_{J}-\cos M\rho|&=2|\sin(2^{-1}M(\rho_{J}+\rho))\sin(2^{-1}M(\rho_{J}-\rho))|\\
&\leq 2|M||\rho_{J}-\rho|\leq2|M|\varepsilon_{J}^{1/4}.
\end{split}
\end{equation*}
The two inequalities above imply that the term defined by \eqref{202308152} is bounded by $2|M|\varepsilon_{J}^{1/4}.$

The discussions above yield
\begin{equation}\label{202401280}
\begin{split}
|\mathcal{I}_{M}^{J}(\Sigma)-\mathcal{R}_{M}^{J}|\leq
2^{-1}\varepsilon_{J}^{3\sigma/4}+2|M|\varepsilon_{J}^{1/4}.
\end{split}
\end{equation}

Now, we give the estimate of the term $\mathcal{R}_{M}^{J}$ defined by \eqref{202308150}. The first inequality in \eqref{202308140} implies that
\begin{equation*}
\begin{split}
\bigcup_{j=0}^{J}\big\{E\in\Gamma_{j}^{(J)}:\ |\sin\xi|\geq\varepsilon_{J}^{1/20}\big\}\subset
\bigcup_{j=0}^{J}\big\{E\in\Gamma_{j}^{(J)}:\ |\sin\xi_{J}|\geq3\varepsilon_{J}^{1/20}/2\big\},
\end{split}
\end{equation*}
which yields
\begin{equation}\label{202308155}
\begin{split}
|\mathcal{R}_{M}^{J}|\leq|\mathcal{W}_{M}^{J}|,
\end{split}
\end{equation}
where
\begin{equation}\label{202308156}
\begin{split}
\mathcal{W}_{M}^{J}(\Sigma):=
\sum_{j=0}^{J}\int_{\{\Gamma_{j}^{(J)}
:|\sin\xi_{J}|>3\varepsilon_{J}^{1/20}/2\}}h_{J}\me^{\mi t\sqrt{E+3}}
\cos M\rho_{J}\cdot\rho_{J}'dE.
\end{split}
\end{equation}

Now, we fix our attention on the set
\begin{equation*}
\begin{split}
\bigcup_{j=0}^{J}\big\{E\in\Gamma_{j}^{(J)}:\ |\sin\xi_{J}|\geq3\varepsilon_{J}^{1/20}/2\big\},
\end{split}
\end{equation*}
which has at most $|\ln\varepsilon_{0}|^{2J^{2}d}$
connected components. Let $(E_{*},E_{**})$ be one of these components, on which $\rho_{J}(E)$ is strictly increasing. Thus, $E(\rho_{J})$ is well-defined.

Note $\rho_{J}'=\xi_{J}'>3^{-1},$ then $|\frac{dE}{d\rho_{J}}|\leq 3.$
Thus, if we set
\begin{equation*}
\begin{split}
F(\rho_{J})=(h_{J}\circ E)(\rho_{J}),
\end{split}
\end{equation*}
then, by the estimates in \eqref{202308114} and \eqref{202308115} we have
\begin{equation*}
\begin{split}
|F|_{C^{1}(\rho_{J}(\Gamma_{0}^{(J)}))}&\leq16/15,\\
|F|_{C^{1}(\rho_{J}(\Gamma_{j+1}^{(J)}))}&\leq10\varepsilon_{j}^{3\sigma},\ 0\leq j\leq J-1, J\geq1.
\end{split}
\end{equation*}
By making the change of variables $E=E(\rho_{J}),$ we have
\begin{equation}\label{202308159+}
\begin{split}
\int_{E_{*}}^{E_{**}}&h_{J}\me^{\mi t\sqrt{E+3}}
\cos M\rho_{J}\cdot\rho_{J}'dE\\
&=\int_{\rho_{J}(E_{*})}^{\rho_{J}(E_{**})}
F(\rho_{J})\me^{\mi t\sqrt{E+3}}
\cos M\rho_{J}d\rho_{J}\\
&=2^{-1}\int_{\rho_{J}(E_{*})}^{\rho_{J}(E_{**})}
F(\rho_{J})\big\{\me^{\mi t(\sqrt{E+3}+t^{-1}M\rho_{J})}
+\me^{\mi t(\sqrt{E+3}-t^{-1}M\rho_{J})}\big\}d\rho_{J}.
\end{split}
\end{equation}

Let $\mathcal{J}\subset (E_{*},E_{**})$ be the subset such that
\begin{equation*}
\begin{split}
-1\leq\cos\xi_{0}\leq\frac{1}{3} \ \text{or}\ \frac{3}{5}\leq\cos\xi_{0}\leq1 \ \text{on} \ \mathcal{J},
\end{split}
\end{equation*}
and
\begin{equation*}
\begin{split}
\frac{1}{3}<\cos\xi_{0}<\frac{3}{5}\  \text{on} \ \ (E_{*},E_{**}) \setminus \ \mathcal{J}.
\end{split}
\end{equation*}
%
Then, by the fact $\rho_{J}\in(-\varepsilon_{J}^{1/4},\pi+\varepsilon_{J}^{1/4}),$ we know that
$\mathcal{J}\subset (E_{*},E_{**}),$ and $(E_{*},E_{**})\setminus \mathcal{J},$ are composed at most by two subinterval intervals (maybe empty), which denote as $\mathcal{J}_{1},\mathcal{J}_{2}$ and $\mathcal{S}_{1},\mathcal{S}_{2},$ respectively.

On $\mathcal{J}_{j},j=1,2,$ \eqref{202308159} implies
$\Big|\frac{d^{2}\sqrt{E+3}}{d\rho_{J}^{2}}\Big|\geq 1/200.$ Then, by applying Van der Corput lemma (Lemma~\ref{VanderCorput2}) with $k=2,$ we get
(for $|t|\geq1$)
\begin{equation*}
\begin{split}
\int_{\rho_{J}(\mathcal{J}_{j})}F(\rho_{J})&\me^{\mi t[\sqrt{E+3}\pm t^{-1}M\rho_{J}]}d\rho_{J}\\
&\leq 18|(200)^{-1}t|^{-1/2} 16(1+\pi) /15<2722|t|^{-1/2}.
\end{split}
\end{equation*}

On $\mathcal{S}_{j},j=1,2,$ \eqref{+202308159} implies
$\Big|\frac{d^{3}\sqrt{E+3}}{d\rho_{J}^{3}}\Big|\geq 1/200.$
By applying Van der Corput lemma (Lemma~\ref{VanderCorput2}) with $k=3,$ we get
(for $|t|\geq1$)
\begin{equation*}
\begin{split}
\int_{\rho_{J}(\mathcal{S}_{j})}F(\rho_{J})&\me^{\mi t[\sqrt{E+3}\pm t^{-1}M\rho_{J}]}d\rho_{J}\\
&\leq 18|(200)^{-1}t|^{-1/3} 16(1+\pi) /15<1089|t|^{-1/3}.
\end{split}
\end{equation*}

The two inequalities above imply that for every connected component $(E_{*},E_{**}),$ we have
\begin{equation*}
\begin{split}
\int_{\rho_{J}((E_{*},E_{**}))}F(\rho_{J})\me^{\mi t[\sqrt{E+3}\pm t^{-1}M\rho_{J}]}d\rho_{J}<7622|t|^{-1/3}.
\end{split}
\end{equation*}
Recalling that there are $|\ln\varepsilon_{0}|^{2J^{2}d}$ connected components
in $\cup_{j=0}^{J}\Gamma_{j}^{(J)},$ then the inequality above, together with
\eqref{202308159+}, and \eqref{202308156}, yield
\begin{equation*}
\begin{split}
\Big|\mathcal{W}_{M}^{J}(\Sigma)\Big|<|\ln\varepsilon_{0}|^{2J^{2}d}
7622\langle t\rangle^{-1/3},
\end{split}
\end{equation*}
which, together with \eqref{202401280} and \eqref{202308155}, yields \eqref{20230815+} for $t>1.$ Obviously, the estimate in \eqref{20230815+} holds trivially for $|t|\leq1.$
\end{proof}

\subsection{Proof of Lemma~3.1.} Fix $t\in\RR^{+},$ and set $J$ large enough such that
\begin{equation}\label{202309222}
\begin{split}
\varepsilon_{J}^{3\sigma/4}\leq\langle t\rangle^{-1/3}.
\end{split}
\end{equation}
Thus,
\begin{equation}\label{202309222+}
\begin{split}
J\geq J_{*}:=\Big[\frac{1}{\ln(1+\sigma)}\ln\big(\frac{4\ln\langle t\rangle}{9\sigma|\ln\varepsilon_{0}|}\big)\Big]+1\leq201\ln\ln(2+\langle t\rangle).
\end{split}
\end{equation}

$\mathbf{Case\ one}: $ $|M|\geq\frac{32}{5}\langle t\rangle^{\frac{4}{3}}.$ We apply the inequality \eqref{202308119}. The first and second terms at $r.h.s.$
of \eqref{202308119} are bounded by $\frac{1}{3}|\ln\varepsilon_{0}|^{2J_{*}^{2}d}\langle t\rangle^{\frac{-4}{3}}$ (take $J=J_{*}$) and $\frac{5}{3}\langle t\rangle^{\frac{-1}{3}},$ respectively.
Summing up and by applying the inequality in \eqref{202309222+} we get the result.

$\mathbf{Case\ two}: $ $|M|<\frac{32}{5}\langle t\rangle^{\frac{4}{3}}.$
We apply the inequality \eqref{20230815+}. Take $J=J_{*},$ then the sum of the first two terms at $r.h.s.$ of \eqref{20230815+} are estimated by $7623|\ln\varepsilon_{0}|^{2J_{*}^{2}d}\langle t\rangle^{\frac{-1}{3}}.$ For the third one just remark that
\begin{equation*}
\begin{split}
2|M|\varepsilon_{J}^{1/4}\leq\frac{64\varepsilon_{J}^{1/8}}{5}\langle t\rangle^{\frac{4}{3}}
(\varepsilon_{J}^{3\sigma/4})^{(6\sigma)^{-1}}<
\frac{1}{3}\langle t\rangle^{\frac{-100}{9}+\frac{4}{3}}<
\frac{1}{3}\langle t\rangle^{-9},
\end{split}
\end{equation*}
where the second inequality is by the inequality in \eqref{202309222} and $\sigma=1/200.$ Summing up and by applying the inequality in \eqref{202309222+} we get the result.
\section{Proof of main results}

\subsection{Proof of Theorem~\ref{theorem1}}
Let $q(t)=\me^{\mi t(H_{\theta}+3)^{1/2}}\phi,$ where $\theta\in\TT^{d}, \phi\in\ell^{1}(\ZZ),$ which solves the first dynamical equation \eqref{mmainequationss}. Set
\begin{equation}\label{202309234+}
\begin{split}
G(E,t):=S(q(t))=\left(
\begin{matrix}
g_{1}(E,t)\\
\\
g_{2}(E,t)
\end{matrix}
\right).
\end{split}
\end{equation}
Thus, for $a.e.\  E\in\Sigma,$ we have
\begin{equation}\label{202309225}
\begin{split}
\left(
\begin{matrix}
g_{1}(E,t)\\
\\
g_{2}(E,t)
\end{matrix}
\right)=\me^{\mi t\sqrt{E+3}}\left(
\begin{matrix}
g_{1}(E,0)\\
\\
g_{2}(E,0)
\end{matrix}
\right).
\end{split}
\end{equation}

The estimate in \eqref{202308121} implies, for all $n\in\ZZ,$
\begin{equation}\label{202309226}
\begin{split}
|q_{n}(t)|\leq|\frac{1}{\pi}\int_{\Sigma}(g_{1}(E)\mathcal{K}_{n}(E)
+g_{2}(E)\mathcal{J}_{n}(E))\rho'dE|
+\varepsilon_{0}^{\sigma^{2}/10}\|q\|_{\ell^{\infty}}.
\end{split}
\end{equation}
Traditional calculations shows
\begin{equation}\label{202309227}
\begin{split}
\int_{\Sigma}&(g_{1}(E,t)\mathcal{K}_{n}(E)
+g_{2}(E,t)\mathcal{J}_{n}(E))\rho'dE\\
&=\int_{\Sigma}\me^{\mi t\sqrt{E+3}}(g_{1}(E,0)\mathcal{K}_{n}(E)
+g_{2}(E,0)\mathcal{J}_{n}(E))\rho'dE\\
&=\int_{\Sigma}\me^{\mi t\sqrt{E+3}}\sum_{m\in\ZZ}\phi_{m}(\sum_{m_{*},n*}
\beta_{m,m_{*}}\beta_{n,n_{*}}\cos(m_{*}-n_{*})\rho)\rho'dE.
\end{split}
\end{equation}

\begin{lemma}\label{lemma4.1}
Assume that $|P|_{r}=\varepsilon_{0}\leq\varepsilon_{*}$ with $\varepsilon_{*}$ in Theorem~\ref{theorem1}. For any $m,m_{*},n,n_{*}$ and $\forall t\in\RR,$
\begin{equation*}
\begin{split}
\Big|\int_{\Sigma}\me^{\mi t\sqrt{E+3}}(
\beta_{m,m_{*}}\beta_{n,n_{*}}\cos(m_{*}-n_{*})\rho)\rho'dE\Big|
\leq7624\langle t\rangle^{-1/3}|\ln\varepsilon_{0}|^{a(\ln\ln(2+\langle t\rangle))^{2}d}.
\end{split}
\end{equation*}
\end{lemma}
\begin{proof}
Apply Lemma~\ref{mainlemma} by setting $h=\beta_{m,m_{*}}\beta_{n,n_{*}}, M=m_{*}-n_{*}$ and $h_{J}=\beta_{m,m_{*}}^{J}\beta_{n,n_{*}}^{J}.$ The result immediately follows.
\end{proof}

$\mathbf{End \ of \ the \ proof \ Theorem~\ref{theorem1}}.$ According to \eqref{202309227} and Lemma~\ref{lemma4.1}, we get, for every $n\in\ZZ,$
\begin{equation*}
\begin{split}
\Big|\int_{\Sigma}(g_{1}(E,t)\mathcal{K}_{n}(E)
+g_{2}(E,t)\mathcal{K}_{n}(E))\rho'dE\Big|
\leq\frac{68616|\ln\varepsilon_{0}|^{a(\ln\ln(2+\langle t\rangle))^{2}d}}{\langle t\rangle^{1/3}}|q(0)|_{\ell^{1}},
\end{split}
\end{equation*}
which, together with \eqref{202309226}, yields
\begin{equation*}
\begin{split}
|q(t)|_{\ell^{\infty}}&\leq\frac{68616}{\pi(1-\varepsilon_{0}^{\sigma^{2}/10})}
\langle t\rangle^{-1/3}|\ln\varepsilon_{0}|^{a(\ln\ln(2+\langle t\rangle))^{2}d}|q(0)|_{\ell^{1}}\\
&\leq22872\langle t\rangle^{-1/3}|\ln\varepsilon_{0}|^{a(\ln\ln(2+\langle t\rangle))^{2}d}|q(0)|_{\ell^{1}},
\end{split}
\end{equation*}
or
\begin{equation}\label{2202309229}
\begin{split}
|q(t)|_{\ell^{\infty}}=|\me^{\mi(H_{\theta}+3)^{1/2}t}q(0)|_{\ell^{\infty}}
\leq22872\langle t\rangle^{-1/3}
|\ln\varepsilon_{0}|^{a(\ln\ln(2+\langle t\rangle))^{2}d}|q(0)|_{\ell^{1}}.
\end{split}
\end{equation}

Note that the general solution of \eqref{mainequation+} is
\begin{equation*}
\begin{split}
u(t)=\cos(t((H_{\theta}+3)^{1/2}) u(0)
+(H_{\theta}+3)^{-1/2}\sin(t(H_{\theta}+3)^{1/2}) \partial_{t}u(0)
\end{split}
\end{equation*}
where $u(0),u_{t}(0)\in\ell^{1}.$ Then, by the estimate in \eqref{2202309229} and the fact that
\begin{equation}\label{202402065}
\begin{split}
\|(H_{\theta}+3)^{-1/2}\|_{\mathcal{L}(\ell^{\infty},\ell^{\infty})}\leq2,
\end{split}
\end{equation}
we get the dispersive estimate \eqref{mainestimate} in Theorem~\ref{theorem1}.

As for the energy estimate, by summation by part we have
\begin{equation*}
\begin{split}
\sum_{n\in\ZZ}(\Delta u_{n} &-V(\theta+n\omega)u_{n})\partial_{t} u_{n}=
\frac{-1}{2}\partial_{t}\sum_{n\in\ZZ}\big\{(u_{n+1}-u_{n})^{2}
+V(\theta+n\omega)u_{n}^{2}\big\}.
\end{split}
\end{equation*}
Thus, by multiplying $\partial_{t}u_{n}(t)$ both
side of \eqref{mainequation+} and sum in $n$ we have
\begin{equation*}
\begin{split}
\partial_{t}\sum_{n\in\ZZ}\big\{(\partial_{t}u_{n})^{2}+(u_{n+1}-u_{n})^{2}+V(\theta+n\omega)u_{n}^{2}\big\}=0,
\end{split}
\end{equation*}
which, together with the inequalities below
\begin{equation*}
\begin{split}
1-\varepsilon_{0}\leq|V(\theta)|=|1+P(\theta)|\leq1+\varepsilon_{0},\forall \theta\in\TT^{d}
\end{split}
\end{equation*}
implies
\begin{equation*}
\begin{split}
\sum_{n\in\ZZ}u_{n}^{2}&\leq\sum_{n\in\ZZ}V^{-1}(\theta+n\omega)
\big\{(\partial_{t}u_{n})^{2}+(u_{n+1}-u_{n})^{2}+V(\theta+n\omega)u_{n}^{2}\big\}\\
&<(1+2\varepsilon_{0})\sum_{n\in\ZZ}
\big\{(\partial_{t}u_{n}(0))^{2}+(u_{n+1}(0)-u_{n}(0))^{2}+(1+2\varepsilon_{0})u_{n}(0)^{2}\big\}\\
&<(1+2\varepsilon_{0})\sum_{n\in\ZZ}
\big\{(\partial_{t}u_{n}(0))^{2}+(5+2\varepsilon_{0})u_{n}(0)^{2}\big\},
\end{split}
\end{equation*}
thus, the energy estimate holds.
\subsection{Proof of Theorem~\ref{theorem2}}
Assume that $\kappa>5$ and $|P|_{r}=\varepsilon_{0}\leq\varepsilon_{*}$ where $\varepsilon_{*}$ is the one in
Theorem~\ref{theorem1}. Then for any $\zeta\in((\kappa-2)^{-1},1/3),$ the estimate in \eqref{2202309229} implies that there exists
$K(\varepsilon_{0},\zeta)$ such that
\begin{equation}\label{202401284}
\begin{split}
|\me^{\mi(H_{\theta}+3)^{1/2}t}\psi|_{\ell^{\infty}}\leq\langle t\rangle^{-\zeta}K|\psi|_{\ell^{1}}, t\in\RR,\psi\in\ell^{1}.
\end{split}
\end{equation}
Then, we cite the well known lemma below
\begin{lemma}\emph{[\cite{zhao20}(Lemma~5.1)]}\label{lemma5.1}
Let $0<\zeta\leq1$ and $\nu>1$ be fixed, then there existence $C_{1}>0$ such that
\begin{equation}\label{202309234}
\begin{split}
\int_{0}^{t}&\frac{1}{\langle t-s\rangle^{\zeta}}\frac{1}{\langle s\rangle^{\nu}}ds<\int_{0}^{\infty}\frac{1}{\langle t-s\rangle^{\zeta}}\frac{1}{\langle s\rangle^{\nu}}ds\leq \frac{C_{1}}{\langle t\rangle^{\zeta}}.
\end{split}
\end{equation}
\end{lemma}

With the parameters $C_{1}$ and $K$ given above, we set
\begin{equation}\label{202402067}
\begin{split}
\delta_{*}=
2^{-1}(6C_{1}(4K)^{2\kappa-1})^{\frac{-1}{2\kappa}}.
\end{split}
\end{equation}
With the same trick proving the energy estimate in Theorem~\ref{theorem1} and the
hypothesis that $|u(0)|_{\ell^{1}},|\partial_{t}u(0)|_{\ell^{1}}:\leq \delta_{*}\ll1,$ we also get the energy estimate,
we omit the details.

Now, we consider the dispersive estimate. The Duhamel formula shows that the general solution of \eqref{mainequation} is given by
\begin{equation}\label{202310170}
\begin{split}
u(t,\psi,\phi)=T[u](t,\psi,\phi),
\end{split}
\end{equation}
implicitly, where
\begin{equation}\label{2024401300}
\begin{split}
T[u](t,\psi,\phi)&=-\int_{0}^{t}(H_{\theta}+3)^{-1/2}
\sin((t-s)(H_{\theta}+3)^{1/2})F(u)ds\\
&+\cos(t(H_{\theta}+3)^{1/2}) \psi
+(H_{\theta}+3)^{-1/2}\sin(t(H_{\theta}+3)^{1/2}) \phi
\end{split}
\end{equation}
with $u(0)=\psi=(\psi_{n})_{n\in\ZZ},
\partial_{t}u(0)=\phi=(\phi_{n})_{n\in\ZZ}\in\ell^{1}$ and $F(u)=(u_{n}^{2\kappa+1})_{n\in\ZZ}.$

For the nonlinear term $F(\psi)$ defined above, we have the estimates
\begin{equation}\label{202309232}
\begin{split}
|F(\psi)|_{\ell^{\infty}}&\leq |\psi|_{\ell^{\infty}}^{2\kappa+1},\\
|F(\psi)|_{\ell^{1}}&\leq |\psi|_{\ell^{2}}^{2}|\psi|_{\ell^{\infty}}^{2\kappa-1}.
\end{split}
\end{equation}

With the estimates above, we are ready to prove the dispersive estimate.
We will prove that the operator $T$ defined by \eqref{2024401300} is a contraction in an appropriate space and apply the contraction mapping theorem to show the existence of fixed point, which is the solution to \eqref{mainequation}. To this end, we define $\mathcal{L}$ by the space on which $T$ act:
\begin{equation*}
\begin{split}
\mathcal{L}_{\delta_{*}}=\Big\{&u:\RR^{+}\times \mathcal{B}_{\delta_{*}}\rightarrow \ell^{\infty}: u(t,\psi,\phi)\in\ell^{\infty},\forall
(t,\psi,\phi)\in\RR^{+}\times \mathcal{B}_{\delta_{*}}\\
&\ \ |u|_{\ell^{\infty}}^{(\zeta}\leq4K(|\psi|_{\ell^{1}}
+|\phi|_{\ell^{1}})
\Big\},
\end{split}
\end{equation*}
where
\begin{equation*}
\begin{split}
\mathcal{B}_{\delta_{*}}:=\Big\{(\psi,\phi)\in\ell^{1}\times\ell^{1}:
|\psi|_{\ell^{1}},|\phi|_{\ell^{1}}\leq\delta_{*}
\Big\},
\end{split}
\end{equation*}
and
\begin{equation}\label{202402061}
\begin{split}
|u|_{\ell^{\infty}}^{(\zeta} :=\sup_{n\in\ZZ}\sup_{\psi,\phi\in\mathcal{B}_{\delta_{*}},t\in\RR^{+}}
|u_{n}(t,\psi,\phi)|\langle t\rangle^{\zeta}.
\end{split}
\end{equation}
To prove that $T$ is a contraction, we apply the induced metric on $\mathcal{L}_{\delta_{*}}$
\begin{equation}\label{202402062}
\begin{split}
d(u,v)=|u-v|_{\ell^{\infty}}^{(\zeta}.
\end{split}
\end{equation}

With the definitions and estimates above, we are ready to prove the operator $T$ is a contraction on $\mathcal{L}_{\delta_{*}}.$

(Step 1) $T\mathcal{L}_{\delta_{*}}\subset\mathcal{L}_{\delta_{*}}.$
The estimates in \eqref{202402065}-\eqref{202309232} and the energy estimate imply that, for any $u\in\mathcal{L}_{\delta_{*}},$
\begin{equation*}
\begin{split}
\Big|\int_{0}^{t}&\sin((t-s)(H_{\theta}+3)^{1/2})F(u(s))ds\Big|_{\ell^{\infty}}\\
&\leq \int_{0}^{t}\Big|\sin((t-s)(H_{\theta}+3)^{1/2})F(u(s))\Big|_{\ell^{1}}ds\\
&\leq \int_{0}^{t}K\langle t-s\rangle^{-\zeta}|F(u(s))|_{\ell^{1}}ds\\
&\leq \int_{0}^{t}K\langle t-s\rangle^{-\zeta}|u(s)|_{\ell^{\infty}}^{2\kappa-1}
|u(s)|_{\ell^{2}}^{2}ds\\
&\leq 6(4K)^{2\kappa-1}K(|\psi|_{\ell^{1}}+|\phi|_{\ell^{1}})^{2\kappa+1}\int_{0}^{t}\langle t-s\rangle^{-\zeta}\langle s\rangle^{-\zeta(2\kappa-1)}ds\\
&\leq 6C_{1}(4K)^{2\kappa-1}K(|\psi|_{\ell^{1}}
+|\phi|_{\ell^{1}})^{2\kappa+1}\langle t\rangle^{-\zeta},
\end{split}
\end{equation*}
where we have used the fact that $\zeta(2\kappa-1)>1.$

The estimate above, together with \eqref{202402065} and \eqref{2024401300}, yields
\begin{equation*}
\begin{split}
|T[u]|_{\ell^{\infty}}^{(\zeta}&\leq 2K(|\psi|_{\ell^{1}}+|\phi|_{\ell^{1}})\{1+6C_{1}(4K)^{2\kappa-1}
(2\delta_{*})^{2\kappa}\}\\
&\leq4K(|\psi|_{\ell^{1}}+|\phi|_{\ell^{1}}),
\end{split}
\end{equation*}
where the last inequality is ensured by the choice of $\delta_{*}.$ Thus, $T\mathcal{L}_{\delta_{*}}\subset\mathcal{L}_{\delta_{*}}.$

(Step 2) $T$ is a contraction in $\mathcal{L}_{\delta_{*}}.$ For any $u,v\in\mathcal{L}_{\delta_{*}},$ we have
\begin{equation*}
\begin{split}
|T[u]-T[v]|_{\ell^{\infty}}
&\leq 2\int_{0}^{t}\Big|\sin((t-s)(H_{\theta}+3)^{1/2})(F(u(s))-F(v(s)))\Big|_{\ell^{\infty}}ds\\
&\leq 2\int_{0}^{t}K\langle t-s\rangle^{-\zeta}|(F(u)-F(v))|_{\ell^{\infty}}ds\\
&< 2K(16K\delta_{*})^{2\kappa}|u-v|_{\ell^{\infty}}^{(\zeta}\int_{0}^{t}\langle t-s\rangle^{-\zeta}\langle s\rangle^{-\zeta(2\kappa+1)}ds\\
&\leq 2C_{1}K(16K\delta_{*})^{2\kappa}|u-v|_{\ell^{\infty}}^{(\zeta}
\leq\frac{1}{2}|u-v|_{\ell^{\infty}}^{(\zeta},
\end{split}
\end{equation*}
where the third inequality is ensured by the following
\begin{equation*}
\begin{split}
\langle s\rangle^{(2\kappa+1)\zeta}&|u_{n}^{2\kappa+1}-v_{n}^{2\kappa+1}|\\
&\leq(2\kappa+1)\max\{(\langle s\rangle^{\zeta}|u_{n}|)^{2\kappa},(\langle s\rangle^{\zeta}|v_{n}|)^{2\kappa}\}
\langle s\rangle^{\zeta}|u_{n}-v_{v}|\\
&<(16K\delta_{*})^{2\kappa}|u-v|_{\ell^{\infty}}^{(\zeta}.
\end{split}
\end{equation*}
The inequality above yields
\begin{equation*}
\begin{split}
d(T[u],T[v])\leq\frac{1}{2}d(u,v).
\end{split}
\end{equation*}
That is the operator $T$ is a contraction in $d$-distance defined in \eqref{202402062}.

From the contraction mapping theorem we know that the operator $T$ possesses unique fixed point $u\in\mathcal{L}_{\delta_{*}},$ which is the solution of
\eqref{mainequation}.

\section*{Appendix}
\subsection{Appendix A. Van der Corput lemma}
For the convenience of readers, we give here the statement of the Van der Corput
lemma and its corollary which are used in this paper, even though they can be found in many textbooks on Harmonic Analysis (see, e.g., Chapter VIII of \cite{Stein93}).

\begin{lemma}\label{VanderCorput1}
Suppose that $\psi$ is real-valued and $C^{k}$ in $(a, b)$ for some $k\geq2$ and
\begin{equation}\label{202310070}
\begin{split}
|\psi^{(k)}(x)|\geq c,\ \forall x\in(a,b),
\end{split}
\end{equation}
For any $\lambda\in\setminus\{0\},$ we have
\begin{equation*}
\begin{split}
\Big|\int_{a}^{b}e^{i\lambda\psi^{(k)}(x)}dx\Big|\leq (2^{k-1}5-2)|c\lambda|^{-k^{-1}}.
\end{split}
\end{equation*}
\end{lemma}

\begin{lemma}\label{VanderCorput2}
Suppose that $\psi$ is real-valued and $C^{k}$ in $(a, b)$ for some $k\geq2$ with the estimate \eqref{202310070}. Let $h$ be $C^{1}$ in $(a, b).$ Then, for all $\lambda \in\RR\setminus\{0\},$
\begin{equation*}
\begin{split}
\Big|\int_{a}^{b}e^{i\lambda\psi^{(k)}(x)}h(x)dx\Big|\leq (2^{k-1}5-2)|c\lambda|^{-k^{-1}}\big\{|h(b)|+\int_{a}^{b}|h'(x)|\big\}.
\end{split}
\end{equation*}
\end{lemma}

\subsection{Appendix B. Cocycles, rotation number}

Assume now  $A\in C^0(\T, SL(2,\R))$ is homotopic to the
identity. Then there exist $\psi:\T \times \T \to \R$ and $u:\T
\times
\T \to \R^+$ such that $$ A(x) \cdot \left (\begin{matrix} \cos 2 \pi y \\
\sin 2 \pi y \end{matrix} \right )=u(x,y) \left (\begin{matrix} \cos 2 \pi (y+\psi(x,y))
\\ \sin 2 \pi (y+\psi(x,y)) \end{matrix} \right ). $$ The function $\psi$ is
called a {\it lift} of $A$.  Let $\mu$ be any probability measure on
$\T \times \T$ which is invariant by the continuous map $T:(x,y)
\mapsto (x+\omega,y+\psi(x,y))$, projecting over Lebesgue measure on
the first coordinate (for instance, take $\mu$ as any accumulation
point of $\frac {1} {n} \sum_{k=0}^{n-1} T_*^k \nu$ where $\nu$ is
Lebesgue measure on $\T \times \T$). Then the number $$
\rho(\omega,A)=\int \psi d\mu \mod \Z $$ does not depend on the
choices of $\psi$ and $\mu$, and is called the {\it fibered rotation
number} of $(\omega,A)$, see \cite{Johnsonm82} and \cite{Herman83}.

\section*{Acknowledgement}
H. Cheng was supported by National Natural Science Foundation of China (12471178, 12001294) and Tianjin Natural Science Foundation of China (23JCYBJC00730).

\bibliographystyle{alpha}
\bibliography{202402kleingordon}

\end{document}